\documentclass[10pt,a4paper]{amsart}

\theoremstyle{plain}
\newtheorem{theorem}{Theorem}[section]
\newtheorem{lemma}[theorem]{Lemma}
\newtheorem{proposition}[theorem]{Proposition}

\theoremstyle{definition}

\theoremstyle{remark}
\newtheorem{remark}[theorem]{Remark}

\def\bin #1#2 {\left( \matrix { #1 \cr #2 \cr } \right) }

\begin{document}

\title[A lower bound for $K^2_S$]
{A lower bound for $K^2_S$}

\author{Vincenzo Di Gennaro }
\address{Universit\`a di Roma \lq\lq Tor Vergata\rq\rq, Dipartimento di Matematica,
Via della Ricerca Scientifica, 00133 Roma, Italy.}
\email{digennar@axp.mat.uniroma2.it}

\author{Davide Franco }
\address{Universit\`a di Napoli
\lq\lq Federico II\rq\rq, Dipartimento di Matematica e
Applicazioni \lq\lq R. Caccioppoli\rq\rq, P.le Tecchio 80, 80125
Napoli, Italy.} \email{davide.franco@unina.it}

\abstract Let $(S,\mathcal L)$ be a smooth, irreducible,
projective, complex surface, polarized by a very ample line bundle
$\mathcal L$ of degree $d > 35$. In this paper we  prove that
$K^2_S\geq -d(d-6)$. The bound is sharp, and $K^2_S=-d(d-6)$ if
and only if $d$ is even, the linear system $|H^0(S,\mathcal L)|$
embeds $S$  in a smooth rational normal scroll $T\subset \mathbb
P^5$ of dimension $3$, and here, as a divisor, $S$ is linearly
equivalent to $\frac{d}{2}Q$, where $Q$ is a quadric on $T$.



\endabstract
\maketitle

\begin{center}
{\it{Dedicated to Philippe Ellia on his sixtieth birthday.}}
\end{center}

\bigskip
\section{Introduction}

The study of numerical invariants of projective varieties, and of
the relations between them, is a classical subject in Algebraic
Geometry. We refer to \cite{Zak1} and \cite{Zak2} for an overview
on this argument. In this paper we turn our attention to the
self-intersection  of the canonical bundle of a smooth projective
surface $S$. One already knows an upper bound in terms of the
degree of $S$ and of the dimension of the space where $S$ is
embedded \cite{D2}. Now we are going to prove the following {\it
lower} bound:

\begin{theorem}\label{lbound}  Let $(S,\mathcal L)$ be a smooth,
irreducible, projective, complex surface, polarized by a very
ample line bundle $\mathcal L$ of degree $d > 35$. Then:
$$
K^2_S\geq -d(d-6).
$$
The bound is sharp, and the following properties are equivalent.

\medskip
(i) $K^2_S= -d(d-6)$;

\medskip
(ii) $h^0(S,\mathcal L)=6$, and the linear system $|H^0(S,\mathcal
L)|$ embeds $S$ in $\mathbb P^5$ as a scroll with sectional genus
$g=\frac{d^2}{8}-\frac{3d}{4}+1$;

\medskip
(iii) $h^0(S,\mathcal L)=6$, $d$ is even, and the linear system
$|H^0(S,\mathcal L)|$ embeds $S$  in a smooth rational normal
scroll $T\subset \mathbb P^5$ of dimension $3$, and here $S$ is
linearly equivalent to $\frac{d}{2}(H_T-W)$, where $H_T$ is the
hyperplane class of $T$, and $W$ the ruling (i.e. $S$ is linearly
equivalent to an integer multiple of a smooth quadric $Q\subset
T$).
\end{theorem}

This is an inequality in the same vein of the classical
Pl\"ucker-Clebsch formula $$g\leq \frac{1}{2}(d-1)(d-2)$$ for the
genus $g$ of a projective curve of degree $d$. Unfortunately, the
argument we developed does not enable us to state a sharp lower
bound depending on the embedding dimension, like Castelnuovo's
bound, neither to examine the case $d\leq 35$.

\bigskip

\section{Proof of Theorem 1.1}

Put $r+1:=h^0(S,\mathcal L)$. Therefore  $|H^0(S,\mathcal L)|$
embeds $S$ in $\mathbb P^r$. Let $H\subseteq \mathbb P^{r-1}$ be
the general hyperplane section of $S$, so that $\mathcal L\cong
\mathcal O_S(H)$. We denote by $g$ the genus of $H$. If $r=2$ then
$d=1$ and $K^2_S=9>5$. If $r=3$ then $K^2_S=d(d-4)^2>-d(d-6)$ for
$d>5$. Therefore we may assume $r\geq 4$.

\bigskip
{\bf The case $r=4$}.

\smallskip
First we examine the case $r=4$. In this case we only have to
prove that, for $d>35$, one has $K^2_S>-d(d-6)$.

\smallskip
When $r=4$ we have the double point formula (\cite{Hartshorne}, p.
433-434, Example 4.1.3):
\begin{equation}\label{double}
d(d-5)-10(g-1)+12\chi(\mathcal O_S)-2K^2_S=0
\end{equation}
(use the adjunction formula $2g-2=H\cdot(H+K_S)$). Moreover by
Lefschetz Hyperplane Theorem we know that the restriction map
$H^1(S,\mathcal O_S)\to H^1(H,\mathcal O_H)$ is injective. So,
taking into account that
$$\chi(\mathcal O_S)=h^0(S,\mathcal O_S)-h^1(S,\mathcal
O_S)+h^2(S,\mathcal O_S),$$ we get
\begin{equation}\label{cai}
\chi(\mathcal O_S)\geq 1-g.
\end{equation}
By (\ref{double}) and (\ref{cai}) we deduce:
$$
2K^2_S\geq d(d-5)-22(g-1).
$$
Therefore to prove that $K^2_S>-d(d-6)$, it is enough to prove
that
\begin{equation}\label{reduce}
22(g-1)<3d^2-17d.
\end{equation}
First assume that $S$ is not contained in a hypersurface of degree
$s<5$. In this case, since $d> 14$, then by Roth's Lemma \cite{R},
\cite{MR}, we know that  $H$ is not contained in a surface of
degree $<5$ in $\mathbb P^3$. Recall that the arithmetic genus of
an irreducible, reduced, nondegenerate space curve of degree
$d>s^2-s$, not contained in a surface of degree $<s$, is bounded
from above by the Halphen's bound \cite{GP}:
$$
G(3;d,s):=\frac{d^2}{2s}+\frac{d}{2}(s-4)+1-\frac{(s-1-\epsilon)(\epsilon+1)(s-1)}{2s},
$$
where $\epsilon$ is defined by dividing $d-1=ms+\epsilon$, $0\leq
\epsilon\leq s-1$. Since $d>20$, we may apply this bound with
$s=5$, and we have:
$$
g\leq \frac{d^2}{10}+\frac{d}{2}+1.
$$
It follows (\ref{reduce}) as soon as $d>35$.

So we may assume that $S$ is contained in an irreducible and
reduced hypersurface of degree $s\leq 4$. First assume
$s\in\{2,3\}$. In this case one knows that for $d>12$ then $S$ is
of general type (\cite{BF}, p. 213), therefore $\chi(\mathcal
O_S)\geq 1$ (\cite{BV}, Th\'eor\`eme X.4, p. 154). Using  this and
(\ref{double}), we see that a sufficient condition for
$K^2_S>-d(d-6)$ is:
\begin{equation}\label{reduce2}
10(g-1)<3d^2-17d+12.
\end{equation}
If $s=2$ then by Halphen's bound we have $g\leq
\frac{d^2}{4}-d+1$. It follows (\ref{reduce2}) for $d>12$. If
$s=3$ then by Halphen's bound we have $g\leq
\frac{d^2}{6}-\frac{d}{2}+1$, from which it follows
(\ref{reduce2}) for $d>7$.

It remains to consider the case $S$ is contained in an irreducible
and reduced hypersurface of degree $s=4$. In this case we need to
refine previous analysis (in fact when $s=4$ one knows that $S$ is
of general type only for $d>97$ (\cite{BF}, p. 213);  moreover if
one simply  inserts Halphen's bound $g\leq \frac{d^2}{8}+1$ into
(\ref{reduce}), the inequality (\ref{reduce}) is satisfied only
for $d>68$). Now first recall that by (\cite{EP}, Lemme 1) one has
$$
\frac{d^2}{8}-\frac{9d}{8}+1\leq g\leq \frac{d^2}{8}+1.
$$
Hence there exists a rational number $0\leq x\leq 9$   such that
$$
g=\frac{d^2}{8}+d\left(\frac{x-9}{8}\right)+1.
$$
If $0\leq x\leq 6$ then $g\leq \frac{d^2}{8}-\frac{3d}{8}+1$, and
(\ref{reduce}) is satisfied for $d>35$. So we may assume $6< x\leq
9$. By (\cite{D}, Proposition 2, (2.2), (2.3) and proof) we have
$$
\chi(\mathcal O_S)\geq
\frac{d^3}{96}-\frac{d^2}{16}-\frac{5d}{3}-\frac{333}{16}-(d-3)d\left(\frac{9-x}{8}\right)
$$
$$
>\frac{d^3}{96}-\frac{d^2}{16}-\frac{5d}{3}-\frac{333}{16}-\frac{3d(d-3)}{8}=
\frac{d^3}{96}-\frac{7d^2}{16}-\frac{13d}{24}-\frac{333}{16}.
$$
From (\ref{double}) it follows that in order to prove that
$K^2_S>-d(d-6)$, it is enough that
$$
10(g-1)\leq
3d^2-17d+12\left(\frac{d^3}{96}-\frac{7d^2}{16}-\frac{13d}{24}-\frac{333}{16}\right),
$$
i.e. it is enough that
$$
10(g-1)\leq
\frac{d^3}{8}-\frac{9d^2}{4}-\frac{47d}{2}-\frac{999}{4}.
$$
Taking into account that $g\leq \frac{d^2}{8}+1$, one sees that
previous inequality holds true for $d>35$.

This concludes the proof of Theorem \ref{lbound} in the case
$r=4$.

\bigskip
{\bf The case $r\geq 6$}.

\medskip
Now we are going to examine the case $r\geq 6$. Also in this case,
we only have to prove that $K^2_S>-d(d-6)$. We distinguish two
cases, according that the line bundle $\mathcal O_S(K_S+H)$ is
spanned or not.

If $\mathcal O_S(K_S+H)$ is spanned then $(K_S+H)^2\geq 0$,
therefore, taking into account the adjunction formula $2g-2=
H\cdot(H+K_S)$, we get
$$
K^2_S\geq d-4(g-1).
$$
Let
$$
G(r-1;d)=\frac{d^2}{2(r-2)}-\frac{rd}{2(r-2)}+\frac{(r-1-\epsilon)(1+\epsilon)}{2(r-2)}
$$
be the Castelnuovo's bound for the genus of a nondegenerate
integral curve of degree $d$ in $\mathbb P^{r-1}$, that we may
apply to $g$ (here $\epsilon$ is defined by dividing
$d-1=m(r-2)+\epsilon$, $0\leq \epsilon\leq r-3$) (\cite{EH},
Theorem (3.7), p. 87). So we deduce
$$
K^2_S+d(d-6)\geq d-4(G(r-1;d)-1)+d(d-6)
$$
$$
=\frac{1}{r-2}\left[(r-4)d^2-(3r-10)d +2(r+\epsilon^2-\epsilon
r+2\epsilon-3)\right].
$$
Since $r\geq 3$ and $\epsilon\geq 0$, we may write:
$$
(r-4)d^2-(3r-10)d+2(r+\epsilon^2-\epsilon r+2\epsilon-3)
$$
$$
\geq (r-4)d^2-(3r-10)d-2\epsilon r
=d^2(r-4)-(5r-10)d+2rd-2\epsilon r.
$$
Observe that we have $d\geq r-1$ for $S$ is nondegenerate in
$\mathbb P^r$. It follows $2rd-2\epsilon r>0$ because
$\epsilon\leq r-3<d$. Hence, in order to prove that
$K^2_S>-d(d-6)$ it suffices to prove that $(r-4)d^2-(5r-10)d\geq
0$, i.e. that
$$
d\geq \frac{5r-10}{r-4}.
$$
Since $d\geq r-1$, this certainly holds for $r\geq 9$. On the
other hand, an elementary direct computation shows that
$$
(r-4)d^2-(3r-10)d+2(r+\epsilon^2-\epsilon r+2\epsilon-3)>0
$$
holds true also for $6\leq r\leq 8$, $0\leq \epsilon\leq r-3$ and
$d\geq r-1$, and for $r=5$ and $d>5$. Summing up, previous
argument shows that
\begin{equation}\label{red1}
{\text{if $\mathcal O_S(K_S+H)$ is spanned, $r\geq 5$ and $d>5$,
then $K^2_S>-d(d-6).$}}
\end{equation}

Now we assume that $\mathcal O_S(K_S+H)$ is not spanned. In this
case one knows that $S$ is a scroll (\cite{SVdV}, Theorem (0.1)),
i.e. $S$ is a $\mathbb P^1$-bundle over a smooth curve $C$, and
the restriction of $\mathcal O_S(1)$ to a fibre is $\mathcal
O_{\mathbb P^1}(1)$ (either $S$ is isomorphic to $\mathbb P^2$,
but in this case $K^2_S=9$). In particular one has that $g$ is
equal to the genus of $C$, and so we have (\cite{Hartshorne},
Corollary 2.11, p. 374)
$$
K^2_S=8(1-g).
$$
Let
$$
G(r;d)=\frac{d^2}{2(r-1)}-\frac{(r+1)d}{2(r-1)}+\frac{(r-\epsilon)(1+\epsilon)}{2(r-1)}
$$
be the Castelnuovo's bound for the genus of a nondegenerate
integral curve of degree $d$ in $\mathbb P^{r}$ (now $\epsilon$ is
defined by dividing $d-1=m(r-1)+\epsilon$, $0\leq \epsilon\leq
r-2$) (\cite{EH}, Theorem (3.7), p.87). Since $g$ is equal to the
genus of $C$, hence to the irregularity of $S$, by (\cite{GGS},
Lemma 4) we have:
$$
g\leq G(r;d).
$$
Hence we deduce:
$$
K^2_S=8(1-g)\geq 8(1-G(r;d)),
$$
and
$$
K^2_S+d(d-6)\geq 8(1-G(r;d))+d(d-6)=:\psi(r;d),
$$
with
$$
\psi(r,d)=\left(\frac{r-5}{r-1}\right)(d^2-2d)-
\frac{4}{r-1}(-r+2-\epsilon-\epsilon^2+\epsilon r).
$$
Taking into account that the function $d\to d^2-2d$ is increasing
for $d\geq 1$, and that $d\geq r-1$, we have:
$$
\psi(r,d)\geq
\psi(r,r-1)=\frac{1}{r-1}\left(r^3-9r^2+27r-23+4\epsilon+4\epsilon^2-4\epsilon
r\right).
$$
Now we notice:
$$
r^3-9r^2+27r-23+4\epsilon+4\epsilon^2-4\epsilon r\geq
r^3-9r^2+27r-23+4\epsilon^2-4\epsilon r
$$
$$=
r^3-10r^2+27r-23+(r-2\epsilon)^2\geq r^3-10r^2+27r-23,
$$
which is $>0$ for $r\geq 7$. An elementary direct computation
proves that $\psi(r,d)>0$ also for $r=6$ (and $d>4$). This
concludes the proof of Theorem \ref{lbound} in the case $r\geq 6$.

\bigskip
\begin{remark}\label{nota1}
We also remark that for $r=5$ we have
$\psi(5,d)=\epsilon^2-4\epsilon+3$. Since
$$
\epsilon^2-4\epsilon+3=\begin{cases} 3\quad {\text{if
$\epsilon=0$}}\\0 \quad {\text{if $\epsilon\in\{1,3\}$}}\\ -1\quad
{\text{if $\epsilon=2$}},
\end{cases}
$$
taking into account (\ref{red1}), it follows that $K^2_S>-d(d-6)$
holds true also for $r=5$ and $d>5$, unless $S\subset \mathbb P^5$
is a scroll, $K^2_S=8(1-g)$,  and
\begin{equation}\label{G}
g=G(5;d)=\frac{1}{8}d^2-\frac{3}{4}d+\frac{(5-\epsilon)(\epsilon+1)}{8},
\end{equation}
with $d-1=4m+\epsilon$, $0<\epsilon\leq 3$. We will use this fact
in the analysis of the case $r=5$ below.
\end{remark}

\bigskip
{\bf The last case: $r=5$}.

\medskip
In this section we examine the case $r=5$, $S\subset \mathbb P^5$.

By previous remark, we know that for $d>5$ one has
$K^2_S>-d(d-6)$, except when the surface $S$ satisfies the
condition $g=G(5;d)$. Now we are going to prove that these
exceptions are necessarily contained in a smooth rational normal
scroll of dimension $3$. As an intermediate step we prove that
such surfaces are contained in a threefold of degree $\leq 4$
(when $d>30$).

To this purpose, assume that $S$ is as before, and that it is not
contained in a threefold of degree $<5$. By (\cite{CC}, Theorem
(0.2)) we know that if $d>24$ then $H$ is not contained in a
surface of degree $<5$ in $\mathbb P^4$. Then by (\cite{EH},
Theorem (3.22), p. 117) we deduce that for $d>143$ one has
$$
g\leq
G(4;d,5):=\frac{1}{10}d^2-\frac{3}{10}d+\frac{1}{5}+\frac{1}{10}v-\frac{1}{10}v^2+w,
$$
where $v$ is defined by dividing $d-1=5n+v$, $0\leq v\leq 4$, and
$w:=\max\{0, [\frac{v}{2}]\}$ (with the notation of \cite{EH} we
have $\pi_{2}(d,4)=G(4;d,5)$). An elementary computation proves
that
\begin{equation}\label{abs}
G(4;d,5)-G(5;d)<0
\end{equation}
for $d>18$. This is absurd, therefore if $K^2_S\leq -d(d-6)$ and
$d>143$, then $S$ is contained in a threefold of degree $\leq 4$.
In order to prove this also for $30<d<144$ we have to refine
previous analysis. To this aim, first recall that
$$
G(4;d,5)=\sum_{i=1}^{+\infty}(d-h(i)),
$$
where
$$
h(i):=
\begin{cases} 5i-1 \quad {\text{if $1\leq i\leq n$}}\\
d-w \quad {\text{if $i=n+1$}}\\
d\quad {\text{if $i\geq n+2$}}
\end{cases}
$$
(\cite{EH}, p. 119). Let $\Gamma\subset \mathbb P^3$ be the
general hyperplane section of $H$, and let $h_{\Gamma}$ be its
Hilbert function.

Assume first that $h^0(\mathbb P^3,\mathcal I_{\Gamma}(2))\geq 2$.
Then, if $d>4$, by monodromy (\cite{CCD}, Proposition 2.1),
$\Gamma$ is contained in a reduced and irreducible space curve  of
degree $\leq 4$. By (\cite{CC}, Theorem (0.2)) we deduce that, for
$d>20$,  $S$ is contained in a threefold of degree $\leq 4$. Hence
we may assume $h^0(\mathbb P^3,\mathcal I_{\Gamma}(2))\leq 1$.

Assume now $h^0(\mathbb P^3,\mathcal I_{\Gamma}(2))=1$, and
$h^0(\mathbb P^3,\mathcal I_{\Gamma}(3))> 4$. As before, if $d>6$,
by monodromy (\cite{CCD}, Proposition 2.1), $\Gamma$ is contained
in a reduced and irreducible space curve  $X$ of degree
$\deg(X)\leq 6$. Again as before, if $\deg(X)\leq 4$, then $S$ is
contained in a threefold of degree $\leq 4$. So we may assume
$5\leq \deg(X)\leq 6$. By (\cite{CCD}, Proposition 4.1) we know
that, when $d>30$,
$$
h_{\Gamma}(i)\geq h(i) \quad {\text{for any $i\geq 0$.}}
$$
Hence we have (\cite{EH}, Corollary (3.2), p. 84):
$$
g\leq \sum_{i=1}^{+\infty}(d-h_{\Gamma}(i))\leq
\sum_{i=1}^{+\infty}(d-h(i))=G(4;d,5).
$$
Since $g=G(5;d)$, this is absurd for $d>18$ (compare with
(\ref{abs})). If $h^0(\mathbb P^3,\mathcal I_{\Gamma}(2))=1$ and
$h^0(\mathbb P^3,\mathcal I_{\Gamma}(3))=4$, then we have
$$
h_{\Gamma}(1)=4,\quad h_{\Gamma}(2)=9, \quad h_{\Gamma}(3)=16.
$$
Using induction and (\cite{EH}, Corollary (3.5), p. 86) we get for
any $i\geq 4$:
\begin{equation}\label{comp}
h_{\Gamma}(i)\geq\min\{d,\, h_{\Gamma}(i-3)+h_{\Gamma}(3)-1\}
\geq\min\{d,\, h(i-3)+15\}\geq h(i).
\end{equation}
As before, this leads to $g\leq G(4;d,5)$, which is absurd for
$d>18$.

Next assume $h^0(\mathbb P^3,\mathcal I_{\Gamma}(2))=0$, and that
$h^0(\mathbb P^3,\mathcal I_{\Gamma}(3))\leq 1$. Then we have:
$h_{\Gamma}(1)=4$, $h_{\Gamma}(2)=10$, $h_{\Gamma}(3)\geq 19$.
Then a similar computation as before leads to a contradiction if
$d>18$.

Finally assume $h^0(\mathbb P^3,\mathcal I_{\Gamma}(2))=0$, and
$h^0(\mathbb P^3,\mathcal I_{\Gamma}(3))\geq 2$. Then, by
monodromy (\cite{CCD}, Proposition 2.1), $\Gamma$ is contained in
a reduced and irreducible curve $X\subset \mathbb P^3$ of degree
$\deg(X)\leq 9$. By (\cite{CCD}, Proposition 4.1) we may also
assume $\deg(X)\geq 7$. Let $X'\subset \mathbb P^2$ be the general
hyperplane section of $X$. By Castelnuovo's Theory (\cite{EH},
Lemma (3.1), p. 83) we know that:
$$
h_X(i)\geq \sum_{j=0}^{i} h_{X'}(j).
$$
Therefore, taking into account (\cite{EH}, Corollary (3.6), p.
87), we have $h_X(1)\geq 4$, $h_X(2)\geq 9$, $h_X(3)\geq 16$. On
the other hand, since $d>27$, by Bezout's Theorem we have
$h_{\Gamma}(i)=h_X(i)$ for any $1\leq i\leq 3$. Hence we may
repeat the same argument as in (\ref{comp}), obtaining $g\leq
G(4;d,5)$, which is absurd.

Summing up, we proved that if $r=5$, $d>30$ and $K^2_S\leq
-d(d-6)$, then $S$ is a scroll, $K^2_S=8(1-g)$, $g=G(5;d)$,
$d\not\equiv 1\,({\text{mod}}\, 4)$, and $S$ is contained in a
threefold $T\subset \mathbb P^5$ of degree  $\leq 4$.
Unfortunately, assuming  $S$ in not contained in a threefold of
degree $<4$, previous argument does not work. Therefore we need a
different argument to prove that $S$ cannot lie in a threefold of
degree $4$.

To this aim, assume by contradiction that  $S$ is contained in a
threefold of degree $4$. Recall that we are assuming that $S$ is a
scroll, $K^2_S=8(1-g)$, $g=G(5;d)$, $d\not\equiv
1\,({\text{mod}}\, 4)$, and that $d>30$. In particular we have
(compare with (\ref{G})):
\begin{equation}\label{G1}
g\geq \frac{1}{8}d^2-\frac{3}{4}d+1.
\end{equation}
On the other hand, by (\cite{EH}, p. 98-99) we know that
$$
h_{\Gamma}(i)\geq k(i):=
\begin{cases} 4i \quad {\text{if $1\leq i\leq p$}}\\
 d-1\quad {\text{if $i= p+1$ and $q=3$}}\\
d\quad {\text{if $i= p+1$ and $q<3$ or $i\geq p+2$,}}
\end{cases}
$$
where $p$ is defined by dividing $d-1=4p+q$, $0\leq q\leq 3$. It
follows that
$$
g\leq \sum_{i=1}^{+\infty}(d-h_{\Gamma}(i))\leq
\sum_{i=1}^{+\infty}(d-k(i))=G(4;d,4),
$$
with
$$
G(4;d,4)=\frac{1}{8}d^2-\frac{1}{2}d+\frac{3}{8}+\frac{1}{4}q-\frac{1}{8}q^2+t,
$$
where $t=0$ if $0\leq q \leq 2$, and $t=1$ if $q=3$ (with the
notation as in (\cite{EH}, p. 99) we have $G(4;d,4)=\pi_1(d,4)$).
Moreover, since $S$ is a scroll, we also have
$$
\chi(\mathcal O_S)=1-g.
$$
And using the same argument as in the proof of (\cite{D},
Proposition 1, (1.2)), we get:
$$
\chi(\mathcal O_S)=1-g\geq 1+\sum_{i=1}^{d-4}(i-1)(d-k(i))-
(d-4)\left(\sum_{i=1}^{d-4}(d-k(i))-g\right)
$$
$$
=1+\sum_{i=1}^{d-4}(i-1)(d-k(i))- (d-4)\left(G(4;d,4)-g\right).
$$
Hence we have
$$
(d-3)g\leq -\sum_{i=1}^{d-4}(i-1)(d-k(i))+ (d-4)G(4;d,4).
$$
Using (\ref{G1}) we get:
$$
(d-3)\left(\frac{1}{8}d^2-\frac{3}{4}d+1\right)\leq
-\sum_{i=1}^{d-4}(i-1)(d-k(i))+ (d-4)G(4;d,4).
$$
Taking into account that
$$
\sum_{i=1}^{d-4}(i-1)(d-k(i))={\binom{p}{2}}d-8{\binom{p+1}{3}}+tp,
$$
previous inequality is equivalent to:
$$
-d^3+24d^2+(-9q^2+18q-125+72t)d-2q^3+42q^2-70q+174-360t+24tq \geq
0.
$$
This is impossible if $d>24$ (recall that $d-1=4p+q$, $0\leq q\leq
3$, and that $t=0$ for $0\leq q \leq 2$, and that $t=1$ for
$q=3$).

So we proved that  if $d>30$ and $K^2_S\leq -d(d-6)$, then $S$ is
a scroll, $g=G(5;d)$, and it is contained in a threefold $T\subset
\mathbb P^5$ of minimal degree $3$, i.e. in a rational normal
scroll $T\subset \mathbb P^5$ of dimension $3$ and degree $3$
(\cite{H}, p. 51).

First we prove that $T$ is necessarily nonsingular. Suppose not.
Let $L$ be a general hyperplane passing through a singular point
of $T$. Then $H\subset L$ is a curve contained in the surface
$T':=T\cap L$, which is a singular rational normal scroll. Put
$d-1=3p+q$, $0\leq q \leq 2$. Since the divisor class group of
$T'$ is generated by a line of the ruling, then $H$ is residual to
$2-q$ lines of the ruling of $T'$, in a complete intersection of
$T'$ with a hypersurface of degree $p+1$. Therefore $H$ is a.C.M.,
and so also $S$ is. In particular the arithmetic genus of $S$ is
equal to the geometric genus, therefore $\chi(\mathcal
O_S)=1-g\geq 1$, i.e. $g=0$, which is impossible in view of the
inequality $g\geq \frac{1}{8}d^2-\frac{3}{4}d+1$.

To conclude the proof of the Theorem it suffices to prove the
following:
\bigskip
\begin{proposition}\label{lboundscroll}
Let $S\subset \mathbb P^5$ be a nondegenerate, smooth,
irreducible, projective, complex surface  of degree $d\geq 18$,
contained in a smooth rational normal scroll $T$ of dimension $3$.
Then $K^2_S\geq -d(d-6)$. The bound is sharp, and the following
properties are equivalent.

\medskip
(i) $K^2_S= -d(d-6)$;

\medskip
(ii) $S$ is a scroll with sectional genus
$g=\frac{d^2}{8}-\frac{3d}{4}+1$;

\medskip
(iii)  $S$ is
linearly equivalent to $\frac{d}{2}(H_T-W)$, where $H_T$ is the
hyperplane class of $T$, and $W$ the ruling.
\end{proposition}

Before proving this, we need the following lemma:

\bigskip
\begin{lemma}\label{divscroll}  Let $T\subset \mathbb P^5$ be
a nonsingular rational normal scroll of dimension $3$. Let $H_T$
be a  hyperplane section of $T$, and $W$ a plane of the ruling.
Let $\alpha$ and $\beta$ be integer numbers. Then the linear
system $|\alpha H_T+\beta W|$ contains an irreducible,
nonsingular, and nondegenerate surface if and only if $\alpha
>0$, $\alpha+\beta \geq 0$, and $3\alpha+\beta\geq 4$.
\end{lemma}
\begin{proof}[Proof of Lemma \ref{divscroll}]
First assume that $|\alpha H_T+\beta W|$ contains an irreducible,
nonsingular, and nondegenerate surface $S$. Let $T':=T\cap\mathbb
P^4$ be a general hyperplane section of $T$, which is a rational
normal scroll surface in $\mathbb P^4$. Let $H_{T'}$ be a
hyperplane section of $T'$, and $W'$ a line of the ruling of $T'$.
Using the same notation of (\cite{Hartshorne}, Notation 2.8.1, p.
373, Example 2.19.1, p. 381), we have $C_0=H_{T'}-2W'$,
$C_0^2=-e=-1$. Therefore the general hyperplane section of $S$
belongs to the linear system $|\alpha H_{T'}+\beta W'|=|\alpha
C_0+(2\alpha+\beta) W'|$. Taking into account that $S$ is
nondegenerate, then by (\cite{Hartshorne}, Corollary 2.18, p. 380)
we get $\alpha >0$, $\alpha+\beta \geq 0$, and
$\deg(S)=3\alpha+\beta\geq 4$.

Conversely, assume $\alpha >0$ and $\alpha+\beta \geq 0$. Using
the same argument as in the proof of (\cite{D2}, Proposition 2.3),
we see that the linear system $|\alpha H_T+\beta W|$ is non empty,
and base point free. By Bertini's Theorem it follows that its
general member is nonsingular. As for the irreducibility, consider
the exact sequence:
$$
0\to \mathcal O_T((\alpha-1)H_T+\beta W)\to \mathcal O_T(\alpha
H_T+\beta W)\to
$$
$$
\to \mathcal O_T(\alpha H_T+\beta W)\otimes \mathcal O_{T'}=
\mathcal O_{T'}(\alpha H_{T'}+\beta W')\to 0.
$$
Since $K_T\sim -3H_T+W$ then we may write:
$$
(\alpha-1)H_T+\beta W=K_T+(\alpha+2)H_T+(\beta-1)W.
$$
As before, by \cite{D2}, we know that the line bundle $\mathcal
O_T((\alpha+2)H_T+(\beta-1)W)$ is spanned, hence nef. On the other
hand we have
$$
((\alpha+2)H_T+(\beta-1)W)^3=3(\alpha+2)^2(\alpha+\beta +1)>0.
$$
Therefore $\mathcal O_T((\alpha+2)H_T+(\beta-1)W)$ is big and nef.
Then by Kawamata-Viehweg Theorem we deduce
$$
H^1(\mathcal O_T((\alpha-1)H_T+\beta W))=0.
$$
This implies that the linear system $|\alpha H_T+\beta W|$ cut on
$T'$ the complete linear system $|O_{T'}(\alpha H_{T'}+\beta
W')|$, whose general member is irreducible by (\cite{Hartshorne},
Corollary 2.18, p. 380). A fortiori the general member of $|\alpha
H_T+\beta W|$ is irreducible.

Finally we notice that the general $S\in |\alpha H_T+\beta W|$ is
nondegenerate. In fact otherwise we would have $S=H_T$, which is
in contrast with our assumption  $\deg(\alpha H_T+\beta
W)=3\alpha+\beta\geq 4$.
\end{proof}

\begin{proof}[Proof of Proposition \ref{lboundscroll}] Define $m$ and $\epsilon$ by diving
\begin{equation}\label{di}
d-1=3m+\epsilon, \quad 0\leq \epsilon \leq 3.
\end{equation}
Since the Picard group of $T$ is freely generated by the
hyperplane class $H_T$ of $T$ and by the plane $W$ of the ruling,
then there exists an unique integer $a\in\mathbb Z$ such that
$$
S \sim (m+1+a)H_T+(\epsilon+1-3(a+1))W.
$$
By previous lemma, we may restrict our analysis to the range
$$
-m\leq a \leq \frac{1}{2}(m+\epsilon-1).
$$
Taking into account that
$$
K_T\sim -3H_T+W,
$$
from the adjunction formula we get (compare with \cite{D2},
$(0.4)$ and p. 149)
$$
K^2_S=\phi(a)=-6a^3+a^2(-9m+5+3\epsilon)+a(2m(3\epsilon-4)-6\epsilon+10)
$$
$$
+3m^3+m^2(3\epsilon-13)+m(10-6\epsilon)+8.
$$
In the given range this function takes its minimum exactly when
$$
a=a^*:=\frac{1}{2}(m+\epsilon-1)
$$
(see Appendix below). Since $\phi(a^*)=-d(d-6)$, it follows
$K^2_S>-d(d-6)$, except when $d$ is even and
$$
\frac{d}{2}=m+1+a^*=-(\epsilon+1-3(a^*+1)).
$$
In this case we already know that $S$ is a scroll with
$g=\frac{d^2}{8}-\frac{3d}{4}+1$.
\end{proof}

\bigskip
{{{\bf{Appendix.}}}}

\bigskip
With the notation as in the proof of Proposition
\ref{lboundscroll}, consider the function
$$
\phi(a):=-6a^3+a^2(-9m+5+3\epsilon)+a(2m(3\epsilon-4)-6\epsilon+10)
$$
$$
+3m^3+m^2(3\epsilon-13)+m(10-6\epsilon)+8.
$$
We are going to prove that {\it if $d\geq 18$ and $-m\leq a \leq
\frac{1}{2}(m+\epsilon-1)$, then $\phi(a)\geq -d(d-6)$, and
$\phi(a)=-d(d-6)$ if and only if $a=a^*$}. To this purpose we
derive with respect to $a$:
$$
\phi'(a)=-18a^2+2a(-9m+5+3\epsilon)+2m(3\epsilon-4)-6\epsilon+10.
$$
This is a degree $2$ polynomial in the variable $a$, whose
discriminant is:
$$
\Delta =324m^2-936m+216m\epsilon+110+114\epsilon+36\epsilon^2,
$$
which is $>0$ when $m\geq 3$, hence when $d\geq 12$ (compare with
(\ref{di})). Denote by $a_1$ and $a_2$ the real roots of the
equation $\phi'(a)=0$, with $a_1<a_2$, and let $I$ be the open
interval $I=(a_1,a_2)$. Then $\phi'(a)>0$ if and only if $a\in I$.
In particular $\phi(a)$ is strictly increasing for $a\in I$, and
strictly decreasing for $a\not\in I$. Now observe that
$$
\phi(-m)=8, \quad \phi(-m+1)=-9m+17-3\epsilon, \quad \phi(-m+2)=0.
$$
Notice that $-9m+17-3\epsilon\geq -9m+11$ because $0\leq
\epsilon\leq 2$, $-9m+11\geq -9\frac{d-1}{3}+11$ since
$\frac{d-1}{3}\leq m\leq \frac{d-3}{3}$, and $-3(d-1)+11>-d(d-6)$
if $d>7$. So for $a\in\{-m,-m+1,-m+2\}$ we have $K^2_S>-d(d-6)$.
Moreover we have $\phi'(-m+2)=18m+6\epsilon-42>0$ if $d\geq 12$,
and $\phi'(-1)=10m+6m\epsilon-12\epsilon-18>0$ if $d\geq 9$.
Therefore $[-m+2,-1]\subset I$ and so $\phi(a)\geq
\phi(-m+2)=0>-d(d-6)$ for $-m\leq a\leq -1$ and $d>12$. We also
have:
$$
\phi(0)=(m-2)(3m^2-7m+3m\epsilon-4),
$$
which is $\geq 0$ for $m\geq 3$, hence for $d\geq 12$. And
$$
\phi(1)=(m-1)(3m^2-10m+3m\epsilon+3\epsilon-17),
$$
which is $\geq 0$ for $m\geq 5$, hence for $d\geq 18$. Moreover we
have:
$$
\phi'(1)=2-26m+6m\epsilon <0.
$$
Therefore $\phi(a)$ is strictly decreasing for $a\geq 1$. It
follows that in the range $1\leq a \leq
a^*:=\left[\frac{m+\epsilon-1}{2}\right]$, the function $\phi$
takes its minimum exactly when $a=a^*$. Define $p$ and $q$ by
dividing:
$$
m+\epsilon-1=2p+q, \quad 0\leq q\leq 1,
$$
so that $p=a^*$. Notice that $d$ is even if and only if $q=0$. We
have:
$$
\phi(a^*)=\begin{cases}-d(d-6)\quad {\text{if $q=0$}}\\
-\frac{1}{4}d^2+\frac{1}{2}d+\frac{35}{4} \quad {\text{if $q=1.$}}
\end{cases}
$$
Since when $d>5$ we have
$$
-\frac{1}{4}d^2+\frac{1}{2}d+\frac{35}{4}>-d(d-6),
$$
by previous analysis it follows that, for any integer $-m\leq a
\leq \frac{m+\epsilon-1}{2}$, one has $\phi(a)\geq -d(d-6)$, and
$\phi(a)=-d(d-6)$ if and only if $d$ is even and
$a=\frac{m+\epsilon-1}{2}$.

\end{document}